\documentclass{article}

\usepackage{amsmath,amssymb,amsfonts,latexsym}
\usepackage[T1]{fontenc}
\usepackage[english]{babel}
\usepackage[utf8]{inputenc}
\usepackage{mathtools}
\usepackage[all]{xy}
\usepackage{mathrsfs}
\usepackage{enumitem}

\usepackage{authblk}
\usepackage{amsthm}

\usepackage{environ}
\newtheorem{thm}{Theorem}[section]
\newtheorem{proposition}[thm]{Proposition}
\newtheorem{theorem}[thm]{Theorem}
\newtheorem{definition}[thm]{Definition}
\newtheorem{lemma}[thm]{Lemma}
\newtheorem{corollary}[thm]{Corollary}

\theoremstyle{definition}
\newtheorem{examplex}[thm]{Example}
\newenvironment{example}
  {\pushQED{\qed}\examplex}
  {\popQED\endexamplex}

\newcommand{\id}{\text{Id}}

\newcommand{\cech}{\v{C}ech }

\newcommand{\unp}{\text{Unp}}

\usepackage{tikz,tikz-cd}
\usetikzlibrary{matrix}
\usepackage{float}

\usepackage{hyperref}
\usepackage{cleveref}
\usetikzlibrary{shapes,arrows}

\title{Determining homology of an unknown space from a sample}
\author[1]{Morten Brun}
\author[1]{Belén García Pascual}
\author[1]{Lars M. Salbu}
\date{}
\affil[1]{Department of Mathematics, University of Bergen.}

\begin{document}
\maketitle
\begin{abstract}
  The homology of an unknown subspace of Euclidean space can be determined from the intrinsic \cech complex of a sample of points in the subspace, without reference to the ambient Euclidean space. More precisely, given a subspace $X$ of Euclidean space and a sample $A$ of points in $X$, we give conditions for the homology of $X$ to be isomorphic to a certain persistent homology group of the intrinsic \cech complex.
  
\end{abstract}

	{Key words: TDA, Interleavings, Persistence, \cech complex, Reconstruction \par}

\section{Introduction}
In the last decades, Topological Data Analysis (TDA) has emerged as an active field where mathematicians and scientists study the shape of data with the aim of accessing new information. For example, spaces of conformations of proteins have been successfully examined through persistent homology (\cite[IX.2]{edelsbrunner}, \cite{Bio}) and the layers of neural networks can be described using persistent homology of image patches \cite{CarlssonMachineLearning}. A general problem in TDA is to determine topological properties of a space from a sample of points in the space. 
We follow the line of research finding assumptions on the sample and the space allowing to determine the homology of the space. In order to do that, we study reconstruction results in the form of finding a simplicial complex that is homotopy equivalent to the space. 

There are many ways of constructing topological spaces from data, including the Vietoris-Rips complex \cite[p.61]{edelsbrunner}, the $\alpha$-complex \cite[p.70]{edelsbrunner} and the \cech complex \cite[p.60]{edelsbrunner}. Given a metric space $M = (M, d)$ we focus on variations of the \cech complex. If \(X\) and \(Y\) are subspaces of \(M\) and \(\alpha > 0\), the (generalized) \cech complex $\mathscr{C}_Y(X,\alpha)$ consists of the finite subsets $\sigma\subseteq X$ so that the intersection of Y and the open balls in $M$ with radius \(\alpha\) and centers in $\sigma$ is non-empty.

Given a metric space \(M\), 
the problem we consider is to determine topological features of a subspace $X$ of $M$ from a known sample
$A$ of points in $X$
and some geometric assumptions. One geometric assumption on the space $X$ is that it is embedded in an Euclidean space with positive reach.
In a sense, the reach describes how curved the space is. In particular, Riemannian manifolds embedded in an Euclidean space have positive reach (see Federer \cite[Sec. 4]{Federer}). A geometric assumption on the sample $A$ is that the directed Hausdorff distance $\overrightarrow{d_H}(X,A)$ is small. Intuitively this distance is the furthest away in $X$ you can go from $A$, so it is small when the sample is sufficiently dense.

The two assumptions described in the previous paragraph have been used with success. Perhaps the best known example is the result by Niyogi, Smale and Weinberg \cite{Smale}, reconstructing a subspace $X$ of $\mathbb{R}^n$ from the ambient \cech complex $\mathscr{C}_{\mathbb{R}^{n}}(A, \alpha)$:
\theoremstyle{plain}
\newtheorem*{well-known proposition}{Theorem \ref{well-known proposition}}
\begin{well-known proposition}[{\cite[Prop. 3.1]{Smale}}]
	For $X$ a compact submanifold in $\mathbb{R}^{n}$ with positive reach $\tau$, and $A$ a finite subspace in $\mathbb{R}^{n}$ such that $d:=\overrightarrow{d_{H}}(X,A)<\sqrt{\frac{3}{20}}\thinspace \thinspace \tau$, then for all $\alpha \in (2d, \sqrt{\frac{3}{5}} \thinspace \thinspace \tau)$, the geometric realization of the \cech complex $\mathscr{C}_{\mathbb{R}^{n}}(A, \alpha)$ is homotopy equivalent to $X$.
\end{well-known proposition}

A more recent reconstruction result stated by both Kim et al. \cite{Wasserman} and García Pascual \cite{Belen} is the following:

\newtheorem*{mythm}{Theorem \ref{mythm}}
\begin{mythm}[{\cite[Cor.10]{Wasserman}}, {\cite[Thm. 2.16]{Belen}}]
	Let $X \subseteq \mathbb{R}^{n}$ have positive reach $\tau$, and let $A\subseteq \mathbb{R}^{n}$. If $\alpha \in \left(\overrightarrow{d_{H}}(X,A), \tau\right]$, then the geometric realization of $\mathscr{C}_{X}(A, \alpha)$ is homotopy equivalent to $X$.
\end{mythm}
In Theorem \ref{mythm} the radius $\alpha$ is increased compared to the radius in Theorem \ref{well-known proposition}, but the \cech complex in Theorem \ref{mythm} depends on the unknown space $X$ and it is therefore difficult to determine without further information. 

We are going to compare \cech complexes of various combinations of ambient metric space $M$, unknown subspace $X$ of $M$ and known sample $A$ of points in $X$. As it turns out, some of these simplicial complexes are homotopy equivalent by Dowker's Theorem, others are included in each other and some are interleaved. In particular, we use interleavings together with the reconstruction result in Theorem \ref{mythm} to replace the ambient Euclidean space $\mathbb{R}^n$ in Theorem \ref{well-known proposition} by the known sample $A$, forming the intrinsic \cech complex $\mathscr{C}_A(A,\alpha)$. This is our main result:

\newtheorem*{mainthm}{Theorem \ref{mainthm}}
\begin{mainthm}
Let $X\subseteq\mathbb{R}^n$ with positive reach $\tau$, let $A\subseteq X$ such that $d:=\overrightarrow{d_{H}}(X,A)<\tau/3$. If $\alpha\in(2d,\tau-d)$, then for any $k\geq 0$ and any $\epsilon\in(d,\tau-\alpha]$ the $k$-th homology group $H_k(X)$ is isomorphic to the $k$-th persistent homology group Im$\,\phi_k^{\alpha,\alpha+\epsilon}$ where 
    \begin{equation*}
        \phi_k^{\alpha,\alpha+\epsilon}:H_k(\mathscr{C}_A(A,\alpha))\to H_k(\mathscr{C}_A(A,\alpha+\epsilon))
    \end{equation*}
is the map induced by the inclusion $\mathscr{C}_A(A,\alpha)\hookrightarrow\mathscr{C}_A(A,\alpha+\epsilon)$.
\end{mainthm}

Ideas in this line of linking reconstruction and interleavings have appeared before in Chazal et al. \cite[Thm 3.7 and Sec 5]{persistreconstruct}.

The \cech complexes obtained when we have a sample $A$ in an unknown subspace $X$ of a metric space $M$ fit into a commutative diagram presented in Proposition \ref{diagram 1}. We study how similar these \cech complexes are to each other. We therefore apply Dowker's Theorem \cite[Theorem 1]{Dowker}, presented here as Corollary \ref{Equiv}, to find homotopy equivalences between some of these \cech complexes. We then establish interleavings between some of the \cech complexes in this diagram. In essence, on the level of homology, we reverse some of the arrows of the diagram by increasing the radii of the \cech complexes using the directed Hausdorff distance between $X$ and $A$, giving the interleavings of Corollary \ref{cor directed hausdorff}. Ultimately, we find that the $k$-th homology group of $X$ is isomorphic to the $k$-th persistent homology group Im$\,\phi_k^{\alpha,\alpha+\epsilon}$ where $\phi_k^{\alpha,\alpha+\epsilon}:H_k(\mathscr{C}_A(A,\alpha))\to H_k(\mathscr{C}_A(A,\alpha+\epsilon))$ is as presented above in Theorem \ref{mainthm}. 

\subsection {Outline}

In Section 2 we introduce (filtered) simplicial complexes, simplicial maps, \cech complexes and commutative diagrams of inclusions of \cech complexes (Proposition \ref{diagram 1}). In Section 3 we present Dowker's Theorem for \cech complexes (Theorem \ref{DowkerCech}), the homotopy equivalences between the different \cech complexes (Corollary \ref{Equiv}) and the induced isomorphisms between their homology groups (Corollary \ref{diagram 1 hom}). In Section 4 we introduce interleavings of persistence groups and contiguous maps. In Section 5 we obtain $(0,\epsilon)$-interleavings for $\epsilon$ greater than the directed Hausdorff distance between $X$ and $A$ (Corollary \ref{cor directed hausdorff}). In Section 6 we define the reach, we look at previous geometric reconstruction results (Theorem \ref{well-known proposition} and Theorem \ref{mythm}) and we give the main result of this paper for determining the homology of the unknown space $X$ (Theorem \ref{mainthm}). In Section 7 we look at how two concrete examples perform under the results presented throughout this paper. In Section 8 we give our conclusions and further research avenues.

\section{\cech complexes}
The main tool we use to find topological features of some unknown space by looking at a known sample, is the \cech complex. In this section we give the basic preliminaries of simplicial complexes and \cech complexes, and show that different kinds of \cech complexes fit into a diagram of inclusions. 

An \textbf{(abstract) simplicial complex} is a set $K$ of finite sets that are closed under inclusions, i.e. if $\sigma\in K$ and $\tau\subseteq \sigma$ then $\tau \in K$. An element of $K$ is called a \textbf{simplex}, and the union of all simplices $V(K)=\cup_{\sigma\in K}\sigma$ is called the \textbf{vertex set} of $K$. For two simplicial complexes $K$ and $K'$, a \textbf{simplicial map} $f:K\to K'$ is a map on the vertex sets $f:V(K)\to V(K')$ sending simplices to simplices, namely such that $\sigma\in K$ implies $f(\sigma):=\{f(s)\}_{s\in\sigma}\in K'$.

We denote the \textbf{geometric realization} of an abstract simplicial complex $K$ by $|K|$ and write $|F|:|K_{1}|\to|K_{2}|$ for the continuous map constructed from a given simplicial map $F:K_{1}\rightarrow K_{2}$ (see for example Spanier \cite[3.2.14 and 3.2.21]{Spanier} for more details). Using the customary convention we write the $k$-th homology group of the realization $|K|$ without the bars, namely as $H_k(K)$. 

Let $(M,d)$ be a metric space. Below we define a (generalized) \cech complex for two arbitrary subspaces of $M$. We follow the definition given by Chazal et al. {\cite[Sec. 4.2.3]{chazal}}, with a slight change in terminology.

\begin{definition}[{\cite[Sec. 4.2.3]{chazal}}]Consider a metric space $(M,d)$ with subsets $X,Y\subseteq M$, and let $\alpha>0$ be a positive radius. The \textbf{(generalized) \cech complex} $\mathscr{C}_{Y}(X,\alpha)$ is the simplicial complex
\begin{equation*}
    \mathscr{C}_{Y}(X,\alpha):=\{\sigma \subseteq X \,|\, \exists y \in Y  \text{ such that } \forall x\in\sigma \quad d(x,y)<\alpha\},
\end{equation*}
i.e. a finite subset $\sigma$ of $X$ is a simplex in $\mathscr{C}_{Y}(X,\alpha)$ whenever the intersection of all the open balls of radius $\alpha$ around points in $\sigma$ intersects $Y$.
\end{definition}

Note that $\mathscr{C}_{Y}(X,\alpha)$ is called the \textbf{ambient \cech complex} in \cite{chazal}, but we use that term for the special case where $Y$ equals the ambient space $M$. If $Y=X$ we have what we call the \textbf{intrinsic \cech complex}.

An important idea in topological data analysis is that instead of looking at only one radius to form a \cech complex, we look at all possible radii, giving a filtration. In general a \textbf{filtered simplicial complex} $\mathcal{K}$ is a collection of simplicial complexes $\mathcal{K}=\{K_\alpha\}_{\alpha>0}$ such that we have inclusions $K_\beta \subseteq K_{\alpha}$ whenever $\beta\leq \alpha$. These inclusions induce homomorphisms on $k$-th homology groups $\phi^{\beta,\alpha}_k:H_k(K_\beta)\to H_k(K_{\alpha})$ for any integer $k\geq 0$, and the images $\text{Im}\,\phi^{\beta,\alpha}_k$ are called the \textbf{$k$-th persistent homology groups} for the filtered simplicial complex $\mathcal{K}$ \cite[VII.1]{edelsbrunner}.  We get the \textbf{filtered \cech complex} by varying $\alpha$, denoted by $\mathcal{C}_Y(X):=\{\mathscr{C}_{Y}(X,\alpha)\}_{\alpha>0}$.\\

The main problem we are examining in this paper is to determine the homology of an unknown metric subspace $X\subseteq M$ by only using a known sample set $A\subseteq X$. The setup presented so far gives us several different \cech complexes by substituting the $X$ and $Y$ in the definition with different combinations of $M$, $X$ and $A$. We want to compare these \cech complexes for the purpose of determining the homology of $X$ through them, first noting from the following lemma that they all fit into a commuting diagram of inclusions.
	
\begin{lemma}\label{lemma inclusions}
	Let $(M,d)$ be a metric space, let $A \subseteq X \subseteq M$ and $Y\subseteq M$. For every $\alpha>0$ we have the following two inclusions:
	\begin{enumerate}[label={\roman*)}]
		\item $\mathscr{C}_{Y}(A, \alpha)\hookrightarrow \mathscr{C}_{Y}(X, \alpha)$ where the vertex map is the inclusion $A\hookrightarrow X$, and
		\item $\mathscr{C}_{A}(Y, \alpha)\hookrightarrow \mathscr{C}_{X}(Y, \alpha)$ where the vertex map is the identity $\id_Y$.
	\end{enumerate}
\end{lemma}
\begin{proof}
	\begin{enumerate}[label={\roman*)}]
		\item Let $\sigma\in\mathscr{C}_{Y}(A, \alpha)$, then $\sigma\subseteq A$ with $y\in Y$ such that $d(a,y)<\alpha$ for all $a\in\sigma$. Since $A\subseteq X$, we have $\sigma\subseteq X$ and we can choose the same $y\in Y$ making $\sigma\in\mathscr{C}_{Y}(X, \alpha)$ . 
		\item Let $\sigma\in\mathscr{C}_{A}(Y, \alpha)$, and let $a\in A$ such that $d(y,a)<\alpha$ for all $y\in\sigma\subseteq Y$. Since $A\subseteq X$ we have $a\in X$, and so $\sigma\in\mathscr{C}_{X}(Y, \alpha)$.
	\end{enumerate}
\end{proof}

In the case of an unknown subspace $X\subseteq M$ together with a known sample $A\subseteq X$, we get the following result directly by changing $Y$, $X$, and $A$ in Lemma \ref{lemma inclusions} with suitable combinations of $A$, $X$ and $M$.

\begin{proposition}\label{diagram 1}
	Let $(M,d)$ be a metric space. If $A \subseteq X \subseteq M$, then for any $\alpha>0$ we obtain the following inclusions of \cech complexes and commutative diagram
		
	\begin{center}	
		$
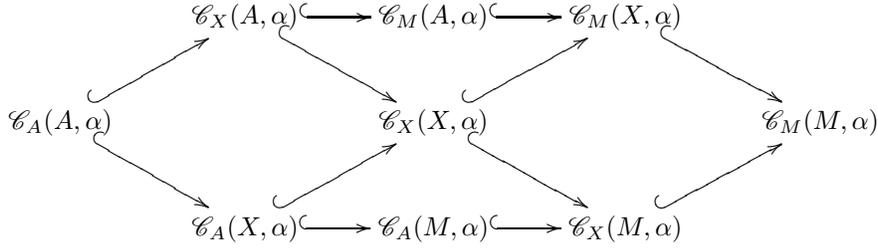
\xymatrix{
			& \mathscr{C}_{X}(A, \alpha) \ar@{^{(}->}[dr]\ar@{^{(}->}[r] & \mathscr{C}_{M}(A, \alpha)\ar@{^{(}->}[r] & \mathscr{C}_{M}(X, \alpha)\ar@{^{(}->}[dr] \\
			\mathscr{C}_{A}(A, \alpha) \ar@{^{(}->}[ur]\ar@{^{(}->}[dr] & 
			& \mathscr{C}_{X}(X, \alpha)\ar@{^{(}->}[ur]\ar@{^{(}->}[dr] & & \mathscr{C}_{M}(M, \alpha)  \\
			& \mathscr{C}_{A}(X, \alpha)\ar@{^{(}->}[ur]\ar@{^{(}->}[r] & \mathscr{C}_{A}(M, \alpha)\ar@{^{(}->}[r] & \mathscr{C}_{X}(M, \alpha)\ar@{^{(}->}[ur]&
		}$\qed
	\end{center}
	
\end{proposition}

This diagram gives a framework for the rest of this paper. In the next section we see how each of the three complexes on the top of the diagram is homotopy equivalent to their respective complex at the bottom, and later we see that we construct simplicial maps in the opposite direction of the ones in the diagram by increasing the radius, leading to interleavings of homology groups.


\section{Dowker's Theorem for \cech complexes}
In this section we give a short demonstration on how Dowker's Theorem \cite[Thm. 1]{Dowker} can be applied to \cech complexes. We reason by the functorial Dowker theorem of Chowdhury and Mémoli \cite[Thm. 3]{chowdhury}.

For two sets $X$ and $Y$, a subset of their product $R\subseteq X\times Y$ is called a \textbf{relation}. The \textbf{Dowker complex} of this relation $R$ is the simplicial complex 
\begin{equation*}
    D(R)=\{\sigma\subseteq X\,|\,\exists y\in Y \text{ such that }(x,y)\in R \,\,\,\forall x\in\sigma\},
\end{equation*}
where simplices are all finite subsets of $X$ whose elements are related to a common element in $Y$.
Dowker's Theorem gives a homotopy equivalence between the Dowker complex of a relation and that of its \textbf{transpose relation} $R^T=\{(y,x)\,|\,(x,y)\in R\}\subseteq Y\times X$ where the same elements are related, but they have changed position. The functorial Dowker theorem \cite[Thm. 3]{chowdhury} states that this homotopy equivalence gives diagrams that commute up to homotopy when looking at inclusions of relations $R'\subseteq R$. Fixing a positive number $\alpha>0$, we can consider any \cech complex of the form $\mathscr{C}_Y(X,\alpha)$ as the Dowker complex of the relation $R_\alpha = \{(x,y)\,|\, d(x,y)<\alpha\}\subseteq X\times Y$ where two elements are related if the distance between them is less than $\alpha$. Moreover, inclusions of spaces give inclusions of relations, leading to the following result.

\begin{theorem}[Dowker's Theorem for \cech complexes]\label{DowkerCech}
    Let $(M,d)$ be a metric space, $0<\beta\leq\alpha$ and $A\times B\subseteq X\times Y\subseteq M\times M$. Then there exist homotopy equivalences $|\Gamma_{A,B}^\beta|:| \mathscr{C}_B(A,\beta)|\rightarrow | \mathscr{C}_A(B,\beta)|$ and $|\Gamma_{X,Y}^{\alpha}|:|\mathscr{C}_Y(X,\alpha)|\rightarrow |\mathscr{C}_X(Y,\alpha)|$ such that the following diagram commutes up to homotopy
    \begin{center}	
		$\xymatrix{
			| \mathscr{C}_B(A,\beta)| \ar[d]^{|\Gamma_{A,B}^\beta|}\ar@{^{(}->}[r]^{}& |\mathscr{C}_Y(X,\alpha)| \ar[d]^{|\Gamma_{X,Y}^{\alpha}|}\\
			| \mathscr{C}_A(B,\beta)|\ar@{^{(}->}[r]^{} & |\mathscr{C}_X(Y,\alpha)|
		}$	
	\end{center}
	where the horizontal maps are inclusions (compositions of maps like in Lemma \ref{lemma inclusions}).\qed
\end{theorem}

Going back to the problem of examining an unknown subspace $X\subseteq M$ from a known sample $A\subseteq X$, we have the following result.
\begin{corollary}\label{Equiv}
	If $A\subseteq X\subseteq M$, then for every $\alpha >0$ we have a diagram 
	\begin{center}	
		$\xymatrix{
			|\mathscr{C}_{X}(A, \alpha)|\ar@{^{(}->}[r] \ar[d]^{\simeq}& |\mathscr{C}_{M}(A, \alpha)|\ar@{^{(}->}[r] \ar[d]^{\simeq}& |\mathscr{C}_{M}(X, \alpha)|\ar[d]^{\simeq}\\
			|\mathscr{C}_{A}(X, \alpha)|\ar@{^{(}->}[r] & |\mathscr{C}_{A}(M, \alpha)|\ar@{^{(}->}[r]& |\mathscr{C}_{X}(M,\alpha)|,
		}$	
	\end{center}
	where each square commutes up to homotopy and the vertical arrows are homotopy equivalences. \qed
\end{corollary}

Homotopy equivalences between topological spaces induce group isomorphisms between their homology (and homotopy) groups (see Hatcher \cite[Cor. 2.11 (and p.342)]{Hatcher}). Moreover, since all maps in the diagram of Proposition \ref{diagram 1} are inclusions, we can apply the version for \cech complexes of Dowker's Theorem \ref{DowkerCech} and get vertical isomorphisms between the homology groups of the top three complexes and the bottom three complexes, as seen in the next result.

\begin{corollary}\label{diagram 1 hom}
	If $A \subseteq X \subseteq M$, then we have a commutative diagram 
	\begin{center}	
		$\xymatrix@C=0.8em{
			& H_k(\mathscr{C}_{X}(A, \alpha)) \ar[dd]^\cong\ar[dr]\ar[r] & H_k(\mathscr{C}_{M}(A, \alpha))\ar[r]\ar@{{-}}[d]^\cong & H_k(\mathscr{C}_{M}(X, \alpha))\ar[dr]\ar[dd]^\cong \\
			H_k(\mathscr{C}_{A}(A, \alpha)) \ar[ur]\ar[dr] & 
			& H_k(\mathscr{C}_{X}(X, \alpha))\ar[d]\ar[ur]\ar[dr] & & H_k(\mathscr{C}_{M}(M, \alpha))  \\
			& H_k(\mathscr{C}_{A}(X, \alpha))\ar[ur]\ar[r] & H_k(\mathscr{C}_{A}(M, \alpha))\ar[r] & H_k(\mathscr{C}_{X}(M, \alpha))\ar[ur]&
		}$
	\end{center}
	for all $k\geq 0$ and all $\alpha>0$, 	where all non-vertical maps are induced by inclusions and the vertical maps are group isomorphisms.\qed
\end{corollary}
We get an analogous result to Corollary \ref{diagram 1 hom} by exchanging homology with homotopy. In the next section we give the basic framework of persistence groups and interleavings, and later we look at geometric conditions of our space $X$ and sampling $A$ that make the filtered versions of maps in the above diagram part of interleavings. These are ultimately used when calculating the homology of the unknown space $X$.

\section{Persistence groups and interleavings}
Having seen that the top and bottom \cech complexes in the diagram in Proposition \ref{diagram 1} are homotopy equivalent, we now want to compare how similar the rest of the \cech complexes in the diagram are to each other. We use a common tool of comparing filtered simplicial complexes, namely interleavings.

\begin{definition}
	A \textbf{persistence group} $G$ is a collection of groups $\{G_\alpha\}_{\alpha>0}$ together with group homomorphisms $\phi^G_{\beta,\alpha}: G_{\beta} \to G_\alpha$ when $0< \beta \leq \alpha $ such that $\phi^G_{\alpha, \alpha} = \text{Id}_{G_\alpha}$ and $ \phi^G_{\gamma, \alpha} = \phi^G_{\beta, \alpha} \circ \phi^G_{\gamma, \beta}$ whenever $0<\gamma \leq \beta \leq \alpha $.
\end{definition}
Categorically this is the same as a functor from the totally ordered set $(0,\infty)$ to the category of groups. We are mostly interested in persistence groups that are the $k$-th homology groups of filtered \cech complexes, namely $G_{\alpha}= H_{k}(\mathscr{C}_{X}(A, \alpha))$ for some dimension $k\geq0$. In this case, the group homomorphism $\phi^G_{\beta,\alpha}$ for $\beta\leq \alpha$ is the homomorphism on homology induced by the inclusion $\mathscr{C}_{X}(A, \beta)\hookrightarrow \mathscr{C}_{X}(A, \alpha)$. If we take homology with coefficients over a field, these persistence groups can be decomposed into interval modules \cite[Thm 1.1]{structurethm} and be visualized in barcodes. All our results hold for homology with any coefficients, and even for homotopy groups, but when talking about barcodes it is understood that we have homology with coefficients over a field. 

\begin{definition}
	For $\delta, \epsilon\geq0$, a \textbf{$(\delta, \epsilon)$-interleaving} $(f_\alpha,g_\alpha)$ between persistence groups $\{A_{\alpha}\}_{\alpha>0}$ and $\{B_{\alpha}\}_{\alpha>0}$ consists of homomorphisms $f_{\alpha}:A_{\alpha}\rightarrow B_{\alpha+\delta}$ and $g_{\alpha}:B_{\alpha}\rightarrow A_{\alpha+\epsilon}$ for every $\alpha>0$ such that the diagrams
	\begin{center}	
		$\xymatrix{
			B_{\alpha}\ar[r]^{g_{\alpha}}\ar[dr]_{\phi^B_{\alpha,\alpha+\epsilon+\delta}}& A_{\alpha+\epsilon}\ar[d]^{f_{\alpha+\epsilon}}& & A_{\alpha}\ar[r]^{f_{\alpha}}\ar[dr]_{\phi^A_{\alpha,\alpha+\delta+\epsilon}}& B_{\alpha+\delta}\ar[d]^{g_{\alpha+\delta}} \\
			& B_{\alpha+\epsilon +\delta } & & & A_{\alpha+\delta +\epsilon }
		}$		
	\end{center}
	commute, and if $\beta\leq\alpha$, then the following two diagrams also commute:
	\begin{center}	
		$\xymatrix{
			B_{\beta}\ar[r]^{g_{\beta}}\ar[d]^{\phi^B_{\beta,\alpha}}& A_{\beta+\epsilon}\ar[d]^{\phi^A_{\beta+\epsilon,\alpha+\epsilon}}& & A_{\beta}\ar[r]^{f_{\beta}}\ar[d]^{\phi^A_{\beta,\alpha}}& B_{\beta+\delta}\ar[d]^{\phi^B_{\beta+\delta,\alpha+\delta}}\\
			B_{\alpha}\ar[r]^{g_{\alpha}}& A_{\alpha+\epsilon} & & A_{\alpha}\ar[r]^{f_{\alpha}}& B_{\alpha+\delta}.
		}$		
	\end{center}
\end{definition}
Considering persistence groups as functors, the last two commuting diagrams correspond to $f$ and $g$ respectively being natural transformations.

When working with homology groups of simplicial complexes, it is useful to know about contiguous maps and the classical result that they are homotopic on the geometric realizations.

\begin{definition}
	Two simplicial maps $F,G:K\rightarrow K'$ are \textbf{contiguous} if for every simplex $\sigma\in K$, the union $F(\sigma)\cup G(\sigma)$ is a simplex in $K'$.
\end{definition}

\begin{lemma}[{\cite[Lemma 3.5.2]{Spanier}}] \label{contiguous lemma}
	If $F,G:K\rightarrow K'$ are contiguous simplicial maps, then $|F|, |G|:|K|\rightarrow|K'|$ are homotopic.
\end{lemma}
With this last result we can can work at the level of simplicial complexes to get interleavings of persistence groups that are homology groups of filtered simplicial complexes. More specifically, if we have two filtered simplicial complexes $\{K_\alpha\}_{\alpha>0}$ and $\{K'_\alpha\}_{\alpha>0}$, then to find a $(\delta,\epsilon)$-interleaving between their homology groups it is sufficient to find simplicial maps $F_\alpha:K_\alpha\to K'_{\alpha+\delta}$ and $G_\alpha:K'_\alpha\to K_{\alpha+\epsilon}$ and check that their composition is contiguous with the inclusion map. By Lemma \ref{contiguous lemma}, this corresponds to commutative diagrams of homology groups that are precisely the ones in the definition of interleaving. We take advantage of this methodology when we compare the different kinds of \cech complexes in the following section.


\section{Interleavings by the directed Hausdorff distance}

We return to our problem of an unknown subspace $X$ in $M$ and a known sample $A$ in $X$. We want to compare the different \cech complexes, in order to ultimately say something about $X$ itself. In this section we look at the directed Hausdorff distance between $X$ and $A$, which in a certain sense represents how dense our sampling is, and see how by restricting it we get interleavings between persistence groups constructed from the homology groups of different filtered \cech complexes. 

To construct these interleavings, we work at the level of \cech complexes by looking at the inclusions in Lemma \ref{lemma inclusions}, finding maps in the opposite direction by increasing the radius, and by applying Lemma \ref{contiguous lemma} for contiguous simplicial maps. All the results for homology in this section also hold for homotopy.

\begin{definition}[\cite{Harker}]
	Let $(X,d)$ be a metric space and $s>0$. A subset $A\subseteq X$ is an \textbf{s-approximation} of $X$ if for every $x\in X$ there exists an $a\in A$ such that $d(x,a)\leq s$, or equivalently if the balls of radius $s$ with centers in $A$ cover $X$.
\end{definition}

The directed Hausdorff distance is the smallest such $s$ making $A$ an $s$-approximation.
\begin{definition}
    The \textbf{directed Hausdorff distance} between two metric subspaces $X, A\subseteq M$ is defined by:
    \begin{equation*}
        \overrightarrow{d_{H}}(X,A):=\sup_{x \in X}\, \inf_{a \in A}\thinspace \thinspace d(x,a).
    \end{equation*}    
\end{definition}

Informally, this is the furthest away from $A$ you can get in $X$. Note in particular that if $A\subseteq X$ is an $s$-approximation, then $\overrightarrow{d_{H}}(X,A)\leq s$. If $A\subseteq X$ is an $s$-approximation of $X$, then there exists a (non-unique) projection map $\Pi:X\to A$ that sends points $x\in X$ to one of the points $\Pi(x)$ in $A$ so that $d(x,\Pi(x))\leq s$ and so that $\Pi|_A=\id_A$. We denote the inclusion map by $\iota:A\hookrightarrow X$.

\begin{proposition}\label{prop 1}
    Let $X,Y\subseteq M$ and let $A\subseteq X$ be an $s$-approximation. For all $k\geq0$ we have:
    \begin{enumerate}[label={\roman*)}]
        \item A $(0,s)$-interleaving $((\iota_\alpha)_*,(\Pi_\alpha)_*)$ between $\{H_{k}(\mathscr{C}_{Y}(A, \alpha))\}_{\alpha>0}$ and     
        
        $\{H_{k}(\mathscr{C}_{Y}(X, \alpha))\}_{\alpha>0}$, where $\iota_\alpha$ and $\Pi_\alpha$ are the maps
        \begin{displaymath}
        \iota_\alpha \colon
                \mathscr{C}_Y(A,\alpha) \hookrightarrow
                \mathscr{C}_Y(X,\alpha)
                \qquad
                \Pi_\alpha \colon \mathscr{C}_Y(X,\alpha) \to
                \mathscr{C}_Y(A,\alpha+s) 
        \end{displaymath}
        induced by the vertex maps \(\iota\) and \(\Pi\) described above. 
        \item A $(0, s)$-interleaving $((Id_{Y})_{\alpha*},(Id_{Y})_{\alpha*})$ between $\{H_{k}(\mathscr{C}_{A}(Y, \alpha))\}_{\alpha>0}$ and  $\{H_{k}(\mathscr{C}_{X}(Y, \alpha))\}_{\alpha>0}$, where the simplicial maps 
            \begin{displaymath}
             (Id_Y)_{\alpha} \colon
                \mathscr{C}_A(Y,\alpha) \hookrightarrow
                \mathscr{C}_X(Y,\alpha)
                \qquad
                (Id_Y)_{\alpha} \colon \mathscr{C}_X(Y,\alpha) \hookrightarrow
                \mathscr{C}_A(Y,\alpha+s).
        \end{displaymath}
        are induced by the identity vertex map on $Y$ denoted by $\id_Y$.
    \end{enumerate}
\end{proposition}

\begin{proof}
	$i)$ The inclusion $\iota_\alpha: \mathscr{C}_{Y}(A, \alpha)\hookrightarrow \mathscr{C}_{Y}(X, \alpha)$ is a simplicial map for all $\alpha>0$ by Lemma \ref{lemma inclusions} $i)$. To see that $\Pi_\alpha:\mathscr{C}_Y(X,\alpha)\to\mathscr{C}_Y(A,\alpha+s)$ is a simplicial map, let $\sigma \in\mathscr{C}_{Y}(X, \alpha)$ with $y\in Y$ so that $d(y,x)<\alpha$ for all $x\in\sigma$. We have that $d(y,\Pi(x))\leq d(y,x)+d(x,\Pi(x))<\alpha+s$ and so $\Pi_\alpha$ is a simplicial map for all $\alpha>0$.
	
	To see that the two following diagrams
    \begin{equation*}
        \begin{tikzcd}
            \mathscr{C}_Y(X,\alpha) \arrow[r,"\Pi_\alpha"]\arrow[dr,hook]&
            \mathscr{C}_Y(A,\alpha+s) \arrow[d,hook]\\
            & \mathscr{C}_Y(X,\alpha+s)
        \end{tikzcd}
        \quad\quad
        \begin{tikzcd}
            \mathscr{C}_Y(A,\alpha) \arrow[r,hook]\arrow[dr,hook] &
            \mathscr{C}_Y(X,\alpha)\arrow[d,"\Pi_\alpha"]\\ & \mathscr{C}_Y(A,\alpha+s)
        \end{tikzcd}
    \end{equation*}	
	commute up to contiguity, consider a simplex $\sigma \in \mathscr{C}_{Y}(X,\alpha)$ with $y\in Y$ so that $d(y,x)<\alpha$ for all $x\in \sigma$. Then $d(y,\Pi(x))\leq d(y, x)+ d(x,\Pi(x))<\alpha + s$ and $d(y,x)<\alpha<\alpha+s$ for any $x\in\sigma$, and in particular $\Pi_{\alpha}(\sigma)\cup \sigma \in\mathscr{C}_{Y}(X, \alpha+s)$. The second diagram commutes using the fact that $\Pi|_A=\id_A$ by construction.\newline
    
    $ii)$ The inclusion $(Id_Y)_\alpha:\mathscr{C}_{A}(Y, \alpha)\hookrightarrow \mathscr{C}_{X}(Y, \alpha)$ is a simplicial map for all $\alpha>0$ by Lemma \ref{lemma inclusions} $ii)$. To see that the inclusion $\mathscr{C}_{X}(Y,\alpha)\hookrightarrow \mathscr{C}_{A}(Y,\alpha+s)$ is well-defined, consider a simplex $\sigma\in \mathscr{C}_{X}(Y,\alpha)$ where $x\in X$ such that $d(x,y)<\alpha$ for all $y\in\sigma$. Then $\Pi(x)\in A$ is such that $d(\Pi(x),y)\leq d(\Pi(x),x)+d(x,y)<s+\alpha$, and so $\sigma\in \mathscr{C}_{A}(Y, \alpha+s)$. The contiguity for this case
    \begin{equation*}
          \begin{tikzcd}
                \mathscr{C}_X(Y,\alpha) \arrow[r,hook]\arrow[dr,hook]&
                \mathscr{C}_A(Y,\alpha+s) \arrow[d,hook]\\
                & \mathscr{C}_X(Y,\alpha+s)
           \end{tikzcd}
           \quad\quad
           \begin{tikzcd}
                \mathscr{C}_A(Y,\alpha) \arrow[r,hook]\arrow[dr,hook] &
                \mathscr{C}_X(Y,\alpha)\arrow[d,hook]\\
              & \mathscr{C}_A(Y,\alpha+s)
          \end{tikzcd}
       \end{equation*}
    
    follows from the fact that all the maps involved are the identity on the vertex set.
\end{proof}

If $\epsilon>\overrightarrow{d_{H}}(X,A)$, then $\epsilon> \inf_{a\in A}d(x,a)$ for all $x\in X$, so there exists an $a\in A$ such that $d(x,a)<\epsilon$ (if not $\epsilon$ would be a bigger lower bound) and thus $A$ is an $\epsilon$-approximation of $X$. This leads to the following corollary.

\begin{corollary}\label{cor directed hausdorff}
Let $X,Y\subseteq M$ and $A\subseteq X$. For any $\epsilon > \overrightarrow{d_{H}}(X,A)$, we have the following for all $k\geq0$:
    \begin{enumerate}[label={\roman*)}]
        \item A $(0,\epsilon)$-interleaving $((\iota_\alpha)_*,(\Pi_\alpha)_*)$ between $\{H_{k}(\mathscr{C}_{Y}(A, \alpha))\}_{\alpha>0}$ and     
        
        $\{H_{k}(\mathscr{C}_{Y}(X, \alpha))\}_{\alpha>0}$, where $\iota_\alpha$ and $\Pi_\alpha$ are the maps
        \begin{displaymath}
        \iota_\alpha \colon
                \mathscr{C}_Y(A,\alpha) \hookrightarrow
                \mathscr{C}_Y(X,\alpha)
                \qquad
                \Pi_\alpha \colon \mathscr{C}_Y(X,\alpha) \to
                \mathscr{C}_Y(A,\alpha+\epsilon) 
        \end{displaymath}
        induced by the vertex maps \(\iota\) and \(\Pi\) described above Proposition \ref{prop 1}.

         \item A $(0, \epsilon)$-interleaving $((Id_{Y})_{\alpha*},(Id_{Y})_{\alpha*})$ between $\{H_{k}(\mathscr{C}_{A}(Y, \alpha))\}_{\alpha>0}$ and  $\{H_{k}(\mathscr{C}_{X}(Y, \alpha))\}_{\alpha>0}$, where the simplicial maps 
            \begin{displaymath}
             (Id_Y)_{\alpha} \colon
                \mathscr{C}_A(Y,\alpha) \hookrightarrow
                \mathscr{C}_X(Y,\alpha)
                \qquad
                (Id_Y)_{\alpha} \colon \mathscr{C}_X(Y,\alpha) \hookrightarrow
                \mathscr{C}_A(Y,\alpha+\epsilon).
        \end{displaymath}
        are induced by the identity vertex map $\id_Y$.\qed
    \end{enumerate}
\end{corollary}
Note that if $A$ is compact, then for all $x\in X$ there exist an $a\in A$ where the infimum is achieved, i.e. where $d(x,a)=\inf_{a\in A} d(x,a)\leq\overrightarrow{d_{H}}(X,A)$. So in this case we also have $(0,\overrightarrow{d_{H}}(X,A))$-interleavings. 

If we look at the special case where $Y=A$, we get interleavings in homology (and homotopy) with the filtered intrinsic \cech complex $\{\mathscr{C}_{A}(A, \alpha)\}_{\alpha>0}$ that just involves points in the known subspace $A$. Looking additionally at the case $Y=X$, we see that the simplicial maps in Corollary \ref{cor directed hausdorff} revert the arrows from the left square of Proposition \ref{diagram 1} in the sense that the diagram
\begin{equation}\label{reversesquare1}
    \begin{tikzcd}
        & \mathscr{C}_{X}(A,\alpha+\epsilon) \arrow[dl,hook] & \\
        \mathscr{C}_{A}(A,\alpha+2\epsilon)\arrow[dd,hook] &
        & \mathscr{C}_{X}(X,\alpha)\arrow[ul,"\Pi_\alpha"]\arrow[dl,hook]\arrow[ddll,hook,bend left=20]\\
        &\mathscr{C}_{A}(X,\alpha+\epsilon)\arrow[ul,"\Pi_\alpha"] &\\
        \mathscr{C}_{X}(X,\alpha+2\epsilon)
    \end{tikzcd}
\end{equation}
commutes up to contiguity, for any $\alpha > 0$ and any $\epsilon>\overrightarrow{d_{H}}(X,A)$.

We end this subsection by noting that the result in Corollary \ref{cor directed hausdorff} is similar to a result from Chazal et al. \cite[Cor. 4.10]{chazal}, where they allow $A\not\subseteq X$ and get an $(\epsilon,\epsilon)$-interleaving if $\epsilon$ is bigger than the (undirected) Hausdorff distance $d_H(X,A)=\max\left\{\overrightarrow{d_{H}}(X,A),\overrightarrow{d_{H}}(A,X)\right\}$.


\section{Determining homology}
In this section we look at previous reconstruction results and see how the framework of interleavings leads to our main result of how to determine the homology groups of an unknown subspace of Euclidean space from the filtered intrinsic \cech complex of a sample. 

The results in this section use assumptions on the directed Hausdorff distance between the unknown space and the sample, and a property called the reach which we define next following the classical text by Federer {\cite[Sec. 4]{Federer}}.

For a metric space $(M,d)$ with a subspace $X\subseteq M$, we write the distance from a point $p\in M$ to such a subspace as $d(X,p)=d(p,X)=\inf_{x\in X} d(x,p)$. 

\begin{definition}[{\cite[Def. 4.1]{Federer}}]\label{unp}
For a subspace $X\subseteq M$, let $\unp(X)$ be the set of all points in $M$ that have a unique nearest point in $X$,
\begin{equation*}
    \unp(X) = \{ p\in M\,|\, \exists ! \, x\in X \text{ such that } d(x', p)\geq d(x,p) \text{ for all } x'\in X\}.
\end{equation*}
\end{definition}
Note that if $p\in \unp(X)$ with nearest point $x$, then $d(x,p)=d(X,p)$. We then have a well-defined projection map $\pi_X:\unp(X)\to X$, sending a point to its unique nearest point in $X$.
\begin{definition}[{\cite[Def. 4.1]{Federer}}]\label{localreach}
 The \textbf{local reach} $\tau(x)$ of a point $x\in X$ is the radius of the biggest ball centered at $x$ that is fully contained in $\unp(X)$, so
 \begin{equation*}
     \tau(x) := \sup \{ r \,|\, B(x,r)\subseteq \unp(X)\},
 \end{equation*}
 where $B(x,r)=\{p\in M\,|\,d(x,p) < r\}$ is the ball of radius $r$ around $x$.
\end{definition}

\begin{definition}[{\cite[Def. 4.1]{Federer}}]\label{reach}
The \textbf{reach} $\tau$ of a subspace $X\subseteq M$ is the smallest of all local reaches,
\begin{equation*}
    \tau:=\inf_{x\in X}\tau(x).
\end{equation*}
\end{definition}

As mentioned in the introduction, the following well-known result shows the reconstruction of a manifold in an Euclidean space ${\mathbb{R}^n}$ with positive reach, from a finite sample $A$ and the ambient \cech complex $\mathscr{C}_{\mathbb{R}^{n}}(A, \alpha)$ that is possible to determine. 

\begin{theorem}[{\cite[Prop. 3.1]{Smale}}]
	\label{well-known proposition}
	For $X$ a compact submanifold in $\mathbb{R}^{n}$ with positive reach $\tau$, and $A$ a finite subspace in $\mathbb{R}^{n}$ such that $d:=\overrightarrow{d_{H}}(X,A)<\sqrt{\frac{3}{20}}\thinspace \thinspace \tau$, then for all $\alpha \in (2d, \sqrt{\frac{3}{5}} \thinspace \thinspace \tau)$, the geometric realization of the \cech complex $\mathscr{C}_{\mathbb{R}^{n}}(A, \alpha)$ is homotopy equivalent to $X$.
\end{theorem}
Note that with this result, we can determine the homology of the unknown space $X$ using the ambient \cech complex that only depends on the sample and the ambient space.

In the introduction, we also stated the following theorem from \cite{Belen} based on the work presented by Jisu Kim et al. in \cite{NThm}. It has also been published in their newer version \cite{Wasserman} with a larger bound. The proof uses the Nerve Theorem (\cite{Hatcher} Corollary 4G.3) together with the fact that the balls with radius in the given interval form a good cover of $X$.

\begin{theorem}[{\cite[Cor.10]{Wasserman}}, {\cite[Thm. 2.16]{Belen}}]
	\label{mythm}
	Let $X \subseteq \mathbb{R}^{n}$ have positive reach $\tau$, and let $A\subseteq \mathbb{R}^{n}$. If $\alpha \in \left(\overrightarrow{d_{H}}(X,A), \tau\right]$, then the geometric realization of $\mathscr{C}_{X}(A, \alpha)$ is homotopy equivalent to $X$.
\end{theorem}
In {\cite[Thm. 2.16]{Belen}} there is an extra condition that $A$ has to be compact, but the proof holds without this as it actually shows that if $A$ is compact, then we can pick $\alpha$ in the closed interval $[\overrightarrow{d_{H}}(X,A), \tau]$.

By setting $A=X$ in the previous theorem and using the fact that $\overrightarrow{d_{H}}(X,X)=0$, we get the following corollary.

\begin{corollary} \label{C X X hom to X}
   Let $X \subseteq \mathbb{R}^{n}$ have positive reach $\tau$. Then for $0<\alpha\leq\tau$, the geometric realization of $\mathscr{C}_{X}(X, \alpha)$ is homotopy equivalent to $X$.\qed
\end{corollary}

Compared to Theorem \ref{well-known proposition}, Theorem \ref{mythm} relaxes some conditions on $X$ and $A$, but in exchange we need to know a lot of information about the unknown space $X$ to determine the homology/homotopy of the \cech complex $\mathscr{C}_{X}(A, \alpha)$. This makes it difficult to use in practice. However, next we use Theorem \ref{mythm} together with the interleavings we found in Section 5 to show that the persistent homology of the filtered intrinsic \cech complex $\mathcal{C}_A(A)=\{\mathscr{C}_A(A,t)\}_{t>0}$ determines the homology of $X$. This \cech complex is constructed only from the known sample $A$, meaning that we have complete information about it.

If we let $\alpha,\epsilon,\epsilon'>\overrightarrow{d_{H}}(X,A)$ such that $\alpha+\epsilon+\epsilon'\leq\tau$, in particular $\overrightarrow{d_{H}}(X,A)<\tau/3$, then by Theorem \ref{mythm} we get a homotopy equivalence $|\mathscr{C}_X(A,r)|\simeq X$ for $r=\alpha$, $r=\alpha+\epsilon$ and $r=\alpha+\epsilon+\epsilon'$. The functorial Nerve theorem \cite[Thm 5,4]{Virk_2021} gives homotopy equivalences between the three \cech complexes for the given $r$'s. Going to the level of homology and using the second interleaving of Corollary \ref{cor directed hausdorff}, we get the following commutative diagram 
\begin{equation}
    \begin{tikzcd}
        H_k(X)\arrow[r,"\cong"]\arrow[d,"="]&
        H_k(\mathscr{C}_X(A,\alpha))\arrow[d,"\cong"]\arrow[r,hook] &
        H_k(\mathscr{C}_A(A,\alpha+\epsilon))\arrow[d,"\phi_k^{\alpha+\epsilon,\alpha+\epsilon+\epsilon'}"]\arrow[dl,swap,"m",twoheadrightarrow] &\\
        H_k(X)\arrow[r,"\cong"]\arrow[d,"="]&
        H_k(\mathscr{C}_X(A,\alpha+\epsilon))\arrow[d,"\cong"]\arrow[r,"l",hook] & H_k(\mathscr{C}_A(A,\alpha+\epsilon+\epsilon'))\arrow[dl,twoheadrightarrow] \\
        H_k(X)\arrow[r,"\cong"]&
        H_k(\mathscr{C}_X(A,\alpha+\epsilon+\epsilon')) &
    \end{tikzcd}
\end{equation}
for any $k\geq 0$. Here $l$ is injective as it is the first factor of a bijection, and similarly $m$ is the last factor and therefore surjective. In particular $l$ is an isomorphism onto the image Im$\,\phi_k^{\alpha+\epsilon,\alpha+\epsilon+\epsilon'}$, which is a $k$-th persistent homology group for the filtered intrinsic \cech complex $\mathcal{C}_A(A)=\{\mathscr{C}_A(A,t)\}_{t>0}$ (see \cite[VII.1]{edelsbrunner}) and can be read directly from the barcodes of $\mathcal{C}_A(A)$ as the lines born before $\alpha+\epsilon$ and dying after $\alpha+\epsilon+\epsilon'$. 

We summarize this in the following theorem considering now a new $\alpha$ as the above $\alpha + \epsilon$ and renaming $\epsilon'$ as simply $\epsilon$.

\begin{theorem}\label{mainthm}
    Let $X\subseteq\mathbb{R}^n$ with positive reach $\tau$, let $A\subseteq X$ such that $d:=\overrightarrow{d_{H}}(X,A)<\tau/3$. If $\alpha\in(2d,\tau-d)$, then for any $k\geq 0$ and any $\epsilon\in(d,\tau-\alpha]$ the $k$-th homology group $H_k(X)$ is isomorphic to the $k$-th persistent homology group Im$\,\phi_k^{\alpha,\alpha+\epsilon}$ where 
    \begin{equation*}
        \phi_k^{\alpha,\alpha+\epsilon}:H_k(\mathscr{C}_A(A,\alpha))\to H_k(\mathscr{C}_A(A,\alpha+\epsilon))
    \end{equation*}
    is the map induced by the inclusion $\mathscr{C}_A(A,\alpha)\hookrightarrow \mathscr{C}_A(A,\alpha+\epsilon)$.   \qed
\end{theorem}
If we pick $\epsilon$ as big as possible we get in particular that $H_k(X)$ is isomorphic to Im$\,\phi_k^{\alpha,\tau}$. Furthermore, if $A$ is finite (and hence compact) we can pick $\alpha = 2d$ by the discussion after Theorem \ref{mythm}, and we have $H_k(X)\cong \text{Im}\,\phi_k^{2d,\tau}$. As all our results before, this is still true when changing homology to homotopy in the statement. 


The arguments in the proof of Theorem \ref{mainthm} are similar to an argument by Chazal and Oudot \cite[Thm. 3.5]{persistreconstruct}, where they look at the weak feature size and identify the homology of $X$ with a persistent homology group for the filtered ambient \cech complex $\mathcal{C}_{\mathbb{R}^n}(A)$. 

We conclude this section stressing that, as stated in Theorem \ref{mainthm}, we can recover the homology of an unknown subspace $X$ of Euclidean space from the filtered intrinsic \cech complex $\mathcal{C}_{A}(A)$ of a sufficiently dense sample $A$ in $X$. 
For large ambient dimension $n$, where the alpha complex is hard to determine, the ambient \cech complex used in Theorem \ref{well-known proposition} is also hard to approach. In this case the intrinsic \cech complex might be easiest to determine. If we take homology with coefficients in a field, we can read off the homology of $X$ from the barcodes of $\mathcal{C}_A(A)$ as the lines born before $\alpha$ and dying after $\alpha+\epsilon$.

\section{Explicit examples}
We now look at two concrete examples of how the interleavings from Corollary \ref{cor directed hausdorff} behave together with the reconstruction results in Section 6. 

\begin{example}[A subspace consisting of two points] \label{example 1} In this first example we work with a finite space, so it is not possible to apply the reconstruction results Theorem \ref{mythm} or Theorem \ref{well-known proposition}, as the directed Hausdorff distance is greater than the reach. We merely get the interleavings from Corollary \ref{cor directed hausdorff} and we check that the obvious homotopy equivalence from Corollary \ref{C X X hom to X} is satisfied.

Consider two distinct points on the real line $X:=\{x_1,x_2\}\subseteq \mathbb{R}$ and let our sample consist of one of the points $A:=\{x_1\}\subseteq X$. Let $\epsilon$ be the directed Hausdorff distance between $X$ and $A$ given by 
$\epsilon:=\overrightarrow{d_{H}}(X,A)=d(x_1,x_2)$. Then the reach $\tau$ is half this distance, that is, $\tau =\epsilon/2$. Calculating the different \cech complexes, we obtain:
\begin{align*}
    \mathscr{C}_{A}(A,\alpha)&=\{\{x_1\}\} \text{ for all } \alpha >0 \\
    \mathscr{C}_{X}(A,\alpha)&=\{\{x_1\}\} \text{ for all } \alpha >0\\
    \mathscr{C}_{A}(X,\alpha) &= 
    \begin{cases}
        \{\{x_1\}\} & 0<\alpha \leq \epsilon \\
        \{\{x_1\}, \{x_2\}, \{x_1,x_2\}\} & \alpha > \epsilon 
    \end{cases}\\
    \mathscr{C}_{X}(X,\alpha) &= 
    \begin{cases}
        \{\{x_1\}, \{x_2\}\} & 0<\alpha \leq \epsilon \\
        \{\{x_1\}\}, \{x_2\}, \{x_1,x_2\}\} & \alpha > \epsilon .
    \end{cases}
\end{align*}
We first note that we have the inclusions from the left square of the diagram in Proposition \ref{diagram 1}, and the homotopy equivalence $|\mathscr{C}_{X}(A,\alpha)|\simeq|\mathscr{C}_{A}(X,\alpha)|$ as expected by Theorem \ref{DowkerCech}. Now if $\pi:X\to A$ is the (only) projection map, then both triangles
\begin{equation*}
    \begin{tikzcd}
        \{\{x_1\}, \{x_2\}\} \arrow[r,"\pi"]\arrow[d,hook]
        & \{\{x_1\}\}\arrow[dl,hook]\\
        \{\{x_1\}, \{x_2\}, \{x_1,x_2\}\} &
    \end{tikzcd}
    \quad\text{and}\quad
    \begin{tikzcd}
        \{\{x_1\}, \{x_2\}, \{x_1,x_2\}\} \arrow[r,"\pi"]\arrow[d,hook]
        & \{\{x_1\}\}\arrow[dl,hook]\\
        \{\{x_1\}, \{x_2\}, \{x_1,x_2\}\} &
    \end{tikzcd}
\end{equation*}
commute up to contiguity, confirming the non-trivial $(0,\epsilon)$-interleavings from Corollary \ref{cor directed hausdorff}, namely the one between $\{H_{k}(\{x_1\})\}$ and $\{H_{k}(\mathscr{C}_{A}(X,\alpha))\}_{\alpha>0}$ and between $\{H_{k}(\{x_1\})\}$ and $\{H_{k}(\mathscr{C}_{X}(X,\alpha))\}_{\alpha>0}$. Finally, if $0<\alpha<\tau<\epsilon$ then $\mathscr{C}_X(X,\alpha)=\{\{x_1\}, \{x_2\}\}$ whose geometric realization obviously is homotopy equivalent to $X=\{x_1,x_2\}$, as predicted by Corollary \ref{C X X hom to X}. 
\end{example}

We proceed to look at an infinite space, in contrast to the finite space from Example \ref{example 1}. It is now possible to determine the homology of the space through the main result in Theorem \ref{mainthm}, which combines interleavings and reconstruction results.

\begin{example}[Uniform sample of the circle] 
Let $X = S^1\subseteq\mathbb{R}^2$ be the unit circle, which has reach $\tau = 1$ as $\unp(S^1)$ is all of $\mathbb{R}^2$ except the center. Furthermore, let $A_q=\{\xi^1,\xi^2,\dots,\xi^q\}\subseteq X$ be a uniform sampling consisting of the $q$ roots of unity, in particular $\xi^{q+1}=\xi^1$. By looking at the problem geometrically, one sees that the directed Hausdorff distance is $\overrightarrow{d_{H}}(S^1, A_q)=\sqrt{2-2\cos(\pi/q)}=2\sin(\pi/2q)$ and if $q\geq10$, then $\overrightarrow{d_{H}}(S^1, A_{q})<1/3=\tau/3$, which is needed to apply Theorem \ref{mainthm}. 

The distance between two consecutive points $\xi^i$ and $\xi^{i+1}$ is $2\sin(\pi/q)$ and in general $d(\xi^{i},\xi^{i+j})=2\sin(j\pi/q)$ whenever $j\leq q/2$. We note that $\sigma\in\mathscr{C}_{A_{q}}(A_{q},\alpha)$ if and only if $\sigma\subseteq \{\xi^{i-j},\dots,\xi^{i-1},\xi^{i},\xi^{i+1},\dots,\xi^{i+j}\}$ for some $j\leq q/2$ such that $d(\xi^{i},\xi^{i+j})<\alpha$ for some $\xi^{i}\in A_{q}$. 

For a concrete example, we let $q=14$ and we pick a radius between $2\overrightarrow{d_{H}}(S^1, A_{14}) \approx 0.45 $ and $2\tau/3=2/3$. For instance we can choose a radius $\alpha = 0.5$. Since $d(\xi^{i},\xi^{i+j})=2\sin (j\pi/14)\leq 0.5$ whenever $j\leq 1$, the \cech complex $\mathscr{C}_{A_{14}}(A_{14},0.5)$ is all subsets in $\{\xi^{i-1},\xi^{i},\xi^{i+1}\}$ for all $\xi^{i}\in A_{14}$. This has a homology group in dimension $1$ generated by the homology class represented by the 1-cycle $c:=\sum ^{14}_{i=1} (\xi^{i}, \xi^{i+1})$. If $\epsilon= 0.5=\tau-\alpha$, then since $2\sin(j\pi/14)<\epsilon+\alpha=1$ for $j\leq 2$, the \cech complex $\mathscr{C}_{A_{14}}(A_{14},1)$ is all subsets in $\{\xi^{i-2},\xi^{i-1},\xi^{i},\xi^{i+1},\xi^{i+2}\}$. We see that $c$ as defined above also represents a generator for the first homology group for this last complex, so the map $\phi^{0.5,1}_1:H_1(\mathscr{C}_{A_{14}}(A_{14},0.5))\to H_1(\mathscr{C}_{A_{14}}(A_{14},1))$ is the identity. 
The homology $H_1(\mathscr{C}_{A_{14}}(A_{14},0.5))$ has degree $1$, and Im$\,\phi^{0.5,1}_1$ is isomorphic to $H_1(S^1)$ by Theorem \ref{mainthm}. So we conclude that the first homology group of the circle $H_1(S^1)$ has degree $1$, as expected.

\end{example}

\section{Conclusion and future research}
We have shown that the homology of an unknown subspace of Euclidean space can be determined from a sample of points in the subspace, without reference to the ambient Euclidean space. More precisely, let $X$ be a subspace of Euclidean space, and let $A$ be a sample of points in $X$. Theorem \ref{mainthm} gives assumptions assuring that the $k$-th homology group of $X$ is isomorphic to the image of $\phi_k^{\alpha,\alpha+\epsilon}:H_k(\mathscr{C}_A(A,\alpha))\to H_k(\mathscr{C}_A(A,\alpha+\epsilon))$ induced by inclusion, for every $k\geq 0$.

In future work we intend to investigate conditions implying that there exists a subspace of the intrinsic \cech complex $\mathscr{C}_A(A,\alpha)$ which is homotopy equivalent to the unknown space $X$. Another line of future research is to replace the ambient Euclidean space by a general metric space $M$, where the first step is to transfer Federer's results on reach \cite{Federer} to $M$.

\bibliographystyle{plainurl}
\bibliography{bibliography}

\end{document}